\documentclass[12pt,a4paper]{article}
\usepackage{fancyhdr}
\usepackage[british]{babel}
\usepackage{amssymb, amsmath, amsthm, amsfonts, enumerate,bm}
\usepackage{amsmath,setspace,scalefnt}
\usepackage[usenames,dvipsnames,svgnames,table]{xcolor}
\usepackage{graphicx,tikz,caption,subcaption}
\usetikzlibrary{patterns}
\usetikzlibrary{shapes}
\usepackage{fullpage}
\usepackage{enumitem,verbatim}
\usepackage{multicol}
\usepackage{float}
\usepackage{cleveref}

\usepackage[margin=2.5cm]{geometry}

\definecolor{gray}{rgb}{0.25, 0.25, 0.25}

\usepackage{booktabs}  

\counterwithin{figure}{section}

\newtheorem{theorem}{Theorem}[section]
\newtheorem{lemma}[theorem]{Lemma}

\newtheorem{cor}[theorem]{Corollary}

\newtheorem*{observation*}{Observation}

\newtheorem*{question*}{Question}
\newtheorem{innerdefinition}{Definition}
\newenvironment{definition*}
  {
   \innerdefinition}
  {\endinnerdefinition}

\theoremstyle{definition}

\theoremstyle{remark}

\usepackage{todonotes}

\newcommand{\ex}{\textrm{ex}}
\newcommand{\EX}{\textrm{EX}}
\newcommand{\spex}{\textrm{spex}}

\title{Spectral extremal problem for the odd prism}
\author{{Xinhui Duan,  Lu Lu\footnote{Corresponding author.}\setcounter{footnote}{-1}\footnote{E-mail address:duanxinhui01@163.com (X. Duan); lulugdmath@163.com (L. Lu).}}\\[2mm]
\small School of Mathematics and Statistics, Central South University,\\
\small Changsha, Hunan, 410083, China\\
}

\date{}
\begin{document}
\maketitle

\begin{abstract}
The spectral Tur\'an number $\spex(n, F)$ denotes the maximum spectral radius $\rho(G)$ of an $F$-free graph $G$ of order $n$. This paper determines $\spex\left(n, C_{2k+1}^{\square}\right)$ for all sufficiently large $n$, establishing the unique extremal graph. Here, $C_{2k+1}^{\square}$ is the odd prism, which is the Cartesian product $C_{2k+1} \square K_2$, where the Cartesian product $G \square F$ has vertex set $V(G) \times V(F)$, and edges between $(u_1,v_1)$ and $(u_2,v_2)$ if either $u_1 = u_2$ and $v_1v_2 \in E(F)$, or $v_1 = v_2$ and $u_1u_2 \in E(G)$.\\

\end{abstract}

\section{Introduction}
In this paper, all graphs considered are undirected, finite and simple. Let $G$ be a graph with vertex set $V(G)$ and edge set $E(G)$. Let $K_n$ be the {\it complete graph} on $n$ vertices, and $K_{s,t}$ be the {\it complete bipartite graph} with parts of sizes $s$ and $t$. We write $C_n$ and $P_n$ for the {\it cycle} and {\it path} on $n$ vertices, respectively. The {\it complement} $\overline{G}$ of $G$ is the graph with vertex set $V(G)$ and edge set $\{uv\mid uv\notin E(G)\}$. For two vertex-disjoint graphs $G$ and $H$, the {\it union} of $G$ and $H$ is the graph $G\cup H$ with vertex set $V(G)\cup V(H)$ and edge set $E(G)\cup E(H)$. In particular, we write $kG$ to denote the vertex-disjoint union of $k$ copies of $G$. The {\it join} of $G$ and $H$, denoted by $G \vee H$, is the graph obtained from $G \cup H$ by adding edges between every vertex of $G$ and every vertex of $H$. For positive integers $n_1,\cdots, n_r$, the complete $r$-partite graph $K_{n_1,\cdots,n_r}$ is defined as $\overline{K_{n_1}\cup \cdots \cup K_{n_r}}$. The $r$-partite Tur\'{a}n graph on $n$ vertices, $T_{n,r}$, is the completer-partite graph $K_{n_1,\cdots,n_r}$ with $\max\{|n_i-n_j| \mid 1\leq i,j\leq r\}\leq 1$. The Cartesian product of graphs $G$ and $F$, denoted by $G\square F$, has vertex set $V(G)\times V(F)$, in which two distinct vertices $(u_1, v_1)$ and $(u_2, v_2)$ are adjacent in $G\square F$ if either $u_1 = u_2$ and $v_1v_2\in E(F)$, or $v_1 = v_2$ and $u_1u_2\in E(G)$. 

A graph $G$ is $F$-free if it contains no subgraph isomorphic to $F$. For a family of graphs $\mathcal{F}$, $G$ is $\mathcal{F}$-free if it is $F$-free for every $F \in \mathcal{F}$. The Tur\'an number $\ex(n,F)$ is defined as
 \[\textrm{ex}(n, F)=\max\left\{e(G) \mid G \in \mathcal{G}_n \text{ and } G \text{ is } F\text{-free}\right\},\] with $\EX(n, F)$ denoting the set of extremal graphs: 
\[\textrm{EX}(n, F) = \left\{G \in \mathcal{G}_n \mid G \text{ is }F\text{-free and }e(G) = \textrm{ex}(n, F)\right\}.\] 
The classical Tur\'an theorem states $\EX(n,K_{r+1}) = \{ T_{n,r} \}$. Since $K_4$ is the tetrahedron, Turán further posed an analogous problem: Determine the Tur\'an number for graphs defined by the vertices and edges of regular polyhedrons. Examples include the cube $Q_8$ and octahedron $O_6$ (Fig. \ref{fig-1-1}). This problem has been extensively studied: see \cite{ES, F1,J,JMY,PS} for $Q_8$; \cite{ES1} for $O_6$; \cite{S1} for the dodecahedron $D_{20}$; and \cite{S2} for the icosahedron $I_{12}$. Related work includes Tur\'an numbers for wheels (Dzido \cite{D1}, Dzido and Jastrzebski \cite{DJ}, Yuan \cite{Y1}) and extremal graphs for cycyle (Ore \cite{Ore}, F\"uredi and Gunderson \cite{FG}).

\begin{figure}[H]
     \centering
     \includegraphics[scale=0.8]{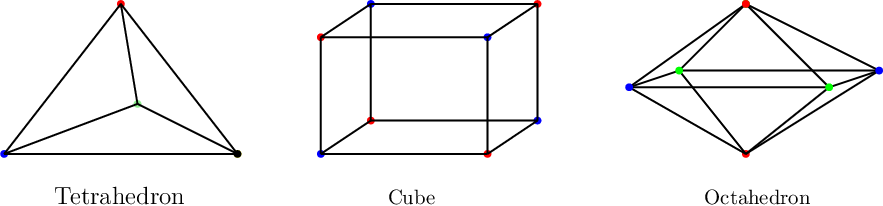}
     \caption{The three regular polyhedrons.}
     \label{fig-1-1}
\end{figure}

In 2022, Brada\u{c}, Janzer, Sudakov and Tomon \cite{BJST2023} studied the Tur\'{a}n number of the grid $P_t\square P_t$, where $P_t$ is the path on $t$ vertices. More precisely, the proved that for $t\geq 2$, threr exist positive numbers $C_1$ and $C_2$ depending only on $t$ such that $C_1n^{3/2}\leq \ex(n,P_t\square P_t)\leq C_2n^{3/2}$. The {\it odd prism} $C_{2k+1}^{\square}$ is defined as $C_{2k+1}^{\square}:=C_{2k+1}\square K_2$. The triangular prism is shown in Fig. \ref{fig-1-2}. He, Li and Feng \cite{HLF2025} consider the Tur\'{a}n number of $C_{2k+1}^{\square}$.
\begin{figure}[H]
     \centering
     \includegraphics[scale=0.7]{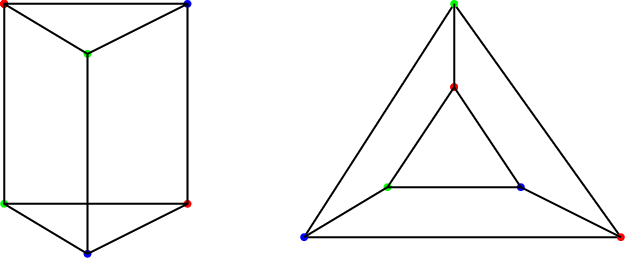}
     \caption{The $C_{3}^{\square}$ graph.}
     \label{fig-1-2}
\end{figure}
\begin{theorem}[\cite{HLF2025}]\label{thm-1-1}
    Let $k\geq 1$ be fixed and $n$ be sufficiently large. Then
    $$
    \ex(n,C_{2k+1}^{\square})=\underset{n_a+n_b=n}{\max}\left\{n_a(1+n_b)+\frac{1}{2}(j^2-3j):j\in\{0,1,2\},j\equiv n_a(\bmod3)   \right\}.
    $$
Moreover, all extremal graphs for $C_{2k+1}^{\square}$ are of the form of a complete bipartite graph $K_{n_a,n_b}$ with an extremal graph for $P_4$ added to the part of size $n_a$.
\end{theorem}

Given a graph $G$, let $A(G)$ be its adjacency matrix, and $\lambda(G)$ its spectral radius. The {\it spectral Tur\'an number} for a  graph $F$ is defined as
 \[\textrm{spex}(n, F)=\max\left\{\lambda(G) \mid G \in \mathcal{G}_n \text{ and } G \text{ is } F\text{-free}\right\},\]
 with $ \textrm{SPEX}(n, F)$ denoting the set of extremal graphs: 
 \[\textrm{SPEX}(n, F) = \left\{G \in \mathcal{G}_n \mid  G \text{ is } F\text{-free and }\lambda(G) = \textrm{spex}(n, F)\right\}.\] 
 This problem has attracted significant attention over the past decades, with results established for numerous specific graphs $F$; comprehensive surveys are found in \cite{CZ,LLF,N1}. Specially, Nikiforov \cite{Nikifrov2008} determined $\spex(n, C_{2k+1})$ for sufficiently large $n$. Nikiforov \cite{Nikiforov2007} and Zhai and Wang \cite{ZW} determined the unique extremal graph with respect to $\spex(n, C_4)$ for odd and even $n$, respectively. Subsequently, Zhai and Lin \cite{ZW} determined the unique extremal graph with respect to $\spex(n, C_6)$. Very recently, Cioab\u{a}, Desai and Tait \cite{CDT} determined the unique extremal graph with respect to $\spex(n, C_{2k})$ for $k \geq 3$ and $n$ large enough.

In this paper, we completely determine $\textrm{SPEX}(n, C_{2k+1}^{\square})$ for large enough $n$.
\begin{theorem}\label{thm-1-2}
    Let $k\geq 1$ be integer. Let $n$ be an sufficiently large integer, and $G$ a $n$-vertex $C_{2k+1}^{\square}$-free graph. Then $$\lambda(G)\leq \lambda(K_1\vee T_{n-1,2}),$$
    with equality holds if and only if $G=K_1\vee T_{n-1,2}$.
\end{theorem}

\section{Preliminary}
Let $G$ be a simple connected graph on $n$ vertices with vertex set $V(G)$ and edge set $E(G)$. For a vertex $v \in V(G)$ and a subset $S \subseteq V(G)$, let $N(v)$ be the neighborhood of $v$, $N_S(v) = N(v) \cap S$, and $N(S) = \bigcup_{u \in S} N(u)$. Denote $d(v) = |N(v)|$ as the degree of $v$, and $d_S(v) = |N_S(v)|$ as the degree of $v$ in $S$. For $V_1, V_2 \subseteq V(G)$, let $E(V_1, V_2)$ denote the set of edges of $G$ with one endpoint in $V_1$ and the other endpoint in $V_2$, and $e(V_1, V_2) = |E(V_1, V_2)|$. Denote $G - S$ as the graph obtained from $G$ by deleting all vertices in $S$ along with their incident edges. This subgraph is also called the subgraph induced by $\overline{S} = V(G) \setminus S$ and is denoted by $G[\overline{S}]$. Particularly, if $S = \left\{u\right\}$, then we write $G - u$ for $G - S$. The adjacency matrix of $G$ is $A(G) = (a_{ij})_{n \times n}$, where $a_{ij} = 1$ if $ij \in E(G)$, and $a_{ij} = 0$ otherwise. The spectral radius of $G$ is the largest eigenvalue of $A(G)$, denoted by $\lambda(G)$. For a connected graph $G$ on $n$ vertices, by the famous Perron-Frobenius theorem, there exists a positive eigenvector $\mathbf{x} = (x_1, \ldots, x_n)^\mathrm{T}$ corresponding to $\lambda(G)$, with the eigenvector normalized such that $\max \left\{x_i \mid 1 \le i \le n\right\} = 1$ ; that is called the {\it Perron vector} of $G$. To obtain our result, we need the following lemmas.

\begin{lemma}[\cite{DKLNTW}]\label{lem-2-1}
     Let $F$ be a graph with chromatic number $\chi(F) = r+1$.  For every $\varepsilon > 0$, there exist $\delta > 0$ and $n_{0}$ such that if $G$ is an $F$-free graph on $n\ge n_0$  vertices with $\lambda ( G) \ge \left ( 1- \frac 1r- \delta \right ) n$, then $G$ can be obtained from $T_{n, r}$ by adding and deleting at most $\varepsilon n^{2}$ edges.
\end{lemma}

Given a graph G, the vertex partition $\Pi: V(G)=V_1\cup V_2\cup\cdots\cup V_k$ is said to be an equitable partition if, for each $u\in V_i,|V_j\cap N(u)|=b_{ij}$ is a constant depending only on $i,j$ $(1\leq i,j\leq k)$. The matrix $B_\Pi=(b_{ij})$ is called the quotient matrix of $G$ with respect to $\Pi$.

\begin{lemma}[\cite{CRS}]\label{lem-2-2}
    Let $\Pi\colon V(G) = V_1\cup V_2\ldots\cup V_k$ be an equitable partition of $G$ with quotient matrix $B_{\Pi}$. Then, the largest eigenvalue of $B_{\Pi}$ is just the spectral radius of $G$.
\end{lemma}

\begin{lemma}[\cite{Niki1}]\label{lem-2-3}
    Let $A$ and $A'$ be the adjacency matrices of two connected graphs $G$ and $G'$ on the same vertex. Suppose that $N_{G}(u)\subsetneq N_{G'}(u)$ for some vertex $u$. If the Perron vector $\textbf{x}$ of $G$ satisfies $\mathbf{x}^TA'\mathbf{x}\geq \mathbf{x}^TA\mathbf{x}$, then $\lambda(G')>\lambda(G)$.
\end{lemma}

\begin{lemma}\label{lem-2-4}
Let $C = x_1x_2\cdots x_{2k+1}x_1$ be an odd cyclic sequence with $x_i \in \{a, b, c, d\}$ for all $1 \le i \le 2k+1$. Then $C$ must contain at least one of the subsequences in the set

$$
\mathcal{F} = \{aa, bb, cc, dd, ad, da, aba, dcd, bdc, cab, cdb, bac\}.
$$

\end{lemma}

\begin{proof}
If $x_i \notin \{a, d\}$ for all $1 \le i \le 2k+1$, then $C$ must contain either the subsequence $bb$ or $cc$. Hence, we may assume without loss of generality that $x_1 = a$.

We proceed by induction on $k$. The base case $k = 1$ can be easily verified by direct inspection.

Now assume the lemma holds for all values less than $k \ge 2$, and suppose for contradiction that $C$ contains none of the subsequences in $\mathcal{F}$. Since $C$ contains none of $\{aa, ad, da\}$, it follows that $x_2, x_{2k+1} \in \{b, c\}$. Furthermore, as $C$ avoids $\{bac, cab\}$, we must have either $x_2 = x_{2k+1} = b$ or $x_2 = x_{2k+1} = c$. Without loss of generality, assume $x_2 = x_{2k+1} = b$ (the other case is symmetric).

Define a new odd cyclic sequence $C' = y y_3 y_4 \cdots y_{2k} y$, where $y = x_2 = x_{2k+1} = b$, and $y_i = x_i$ for all $3 \le i \le 2k$. Since $C$ contains none of the subsequences in $\mathcal{F}$, it follows that the linear sequence $y_3 y_4 \cdots y_{2k}$ also avoids $\mathcal{F}$. Thus, by the inductive hypothesis, $C'$ must contain a subsequence in $\mathcal{F}$. That subsequence must include $y = b$, and hence one of the following cases occurs:
\begin{itemize}
\item[(i)] $y_{2k}y = bb$ or $y y_3 = bb$,
\item[(ii)] $y_{2k} y y_3 = aba$,
\item[(iii)] $y y_3 y_4 \in \{bdc, bac\}$,
\item[(iv)] $y_{2k-1} y_{2k} y \in \{cdb, cab\}$.
\end{itemize}
If (i) occurs, then $x_{2k}x_{2k+1}=bb$ or $x_2x_3=bb$; if (ii) occurs, then $x_{2k}x_{2k+1}x_1=aba$; if (iii) occurs, then $x_2x_3x_4\in\{bdc,bac\}$; if (iv) occurs, then $y_{2k-1}y_{2k}y_{2k+1}\in\{cdb,cab\}$. In all cases, a contradiction arises to the assumption that $C$ contains none of the subsequences in $\mathcal{F}$.

This completes the proof.
\end{proof}
\begin{cor}\label{cor-2-5}
   If we color each vertex of $C_{2k+1}^{\square}$ with red or blue, then there exists either a monochromatic path $P_4=v_1v_2u_1u_2 $,  or a cycle $C_4 = uvwz$ where vertices $u $ and $ v $ are red while $ w,z $ are blue.  
\end{cor}
\begin{proof}
Let $V(C_{2k+1}^{\square})=U\cup V$ where $U=\{u_1,u_2,\ldots,u_{2k+1}\}$ and $V=\{v_1,v_2,\ldots,v_{2k+1}\}$ such that $U$ and $V$ induce two copies of $C_{2k+1}$ and $\{u_1 v_1,u_2v_2,\ldots,u_{2k+1}v_{2k+1}\}$ is a matching. Define an odd-length cyclic sequence $C = x_1\cdots x_{2k+1}$ with $x_i \in \{a,b,c,d\}$ encoding the color pairs:  
\[x_i=\left\{\begin{array}{ll}
 a & \text{if both } u_i \text{ and } v_i \text{are colored red}, \\[1mm]
 b & \text{if } u_i \text{ is colored red and } v_i \text{ is colored blue}, \\[1mm]
 c & \text{if } u_i \text{ is colored blue and } v_i \text{ is colored red}, \\[1mm]
d & \text{if both} u_i\text{ and } v_i \text{ are colored blue}.
\end{array}\right.\]
By Lemma \ref{lem-2-4}, $C$ contains a subsequence in $\mathcal{F}$. It is easy to verify that, if $C$ contains one of $\{aa,dd,aba,dcd,bdc,cab,cdb,bac\}$, then $C_{2k+1}^{\square}$ has a monochromatic $P_4$; if $C$ contains one of $\{aa,dd,aba,dcd,bdc,cab,cdb,bac\}$, then $C_{2k+1}^{\square}$ has a desired $C_4$.
\end{proof}

\section{Proof of Theorem \ref{thm-1-2}}
 For brevity, we always assume $ n $ to be sufficiently large throughout the paper. Let $ G $ be the graph attaining the maximum spectral radius among all $ n $-vertex graphs containing no $ C_{2k+1}^{\square} $, and let $ \bm{x} = (x_1, x_2, \ldots, x_n)^T $ be the Perron vector of $ G $. We will establish a series of lemmas to characterize the structure of $ G $.  
\begin{lemma}\label{lem-3-1}
    $\lambda(G)\geq \lambda(K_{1}\vee T_{n-1,2})\geq n/2.$
\end{lemma}
\begin{proof}
Clearly, $K_{1}\vee T_{n-1,2}$ is $C_{2k+1}^{\square}$-free. Since $e(T_{n-1,2})=\lfloor\frac{(n-1)^2}{4}\rfloor\geq \frac{n^2-2n-2}{4}$, we have 
$$
e(K_{1}\vee T_{n-1,2})=e(T_{n-1,2})+(n-1)\geq\frac{n^2+2n-6}{4}. 
$$
Using the Rayleigh quotient gives
$$
\lambda(G) \ge \lambda(K_{1}\vee T_{n-1,2})\ge \frac{\mathbf{1}^TA(K_{1}\vee T_{n-1,2})\mathbf{1}}{\mathbf{1}^T\mathbf{1}}=\frac{2e(K_{1}\vee T_{n-1,2})}{n}\geq \frac{n}{2}
$$  
for sufficiently large $n$, as desired.
\end{proof}

Next we show that $G$ contains a large maximum cut.
\begin{lemma}\label{lem-3-2}
     For every $\varepsilon>0$ and $n$ sufficiently large,  $e(G)\geq e(T_{n,2})-\varepsilon n^2$. Furthermore, $G$ admits a partition $V(G)=V_1\cup V_2$ such that $e(V_1, V_2)$ attains the maximum, $e(V_1)+e(V_2) \leq \varepsilon n^2$, and $||V_i|-n/2|\leq \frac{3\sqrt{\varepsilon}n}{2}$ for each $i\in[2]$.
\end{lemma}
\begin{proof}
    Since $G$ is $C_{2k+1}^{\square}$-free, by Lemmas \ref{lem-2-1} and \ref{lem-3-1}, for sufficiently large $n$, there exists a partition of $V(G)=U_1\cup U_2$ such that $e(G)\geq e(T_{n,2})-\varepsilon n^2$, $e(U_1)+e(U_2)\leq \varepsilon n^2$, and $\lfloor\frac{n}{2}\rfloor\leq |U_i|\leq \lceil\frac{n}{2}\rceil$ for $i\in \{1,2\}$. Therefore such a partition $V(G)=V_1\cup V_2$ such that the number of crossing edges of $G$ attains the maximum, and thus
     \[e(V_1)+e(V_2)\leq e(U_1)+e(U_2)\leq \varepsilon n^2.\]
     Let $a=||V_1|-\frac{n}{2}|=||V_2|-\frac{n}{2}|$. Then we have 
    $$e(G)=e(V_1,V_2)+e(V_1)+e(V_2)\leq |V_1||V_2|+\varepsilon n^2=\frac{n^2}{4}-a^2+\varepsilon n^2.$$
    Recall that 
    $$e(G)\geq e(T_{n,2})-\varepsilon n^2\geq \frac{n^2-1}{4}-\varepsilon n^2.$$
    It leads to
    $$a\leq \sqrt{2\varepsilon n^2+\frac{1}{4}}\leq \sqrt{\frac{9\varepsilon n^2}{4}}=\frac{3\sqrt{\varepsilon}n}{2},$$
    as desired.
\end{proof}

In what follows, we will always assume $V(G)=V_1\cup V_2$ is the partition characterized in Lemma \ref{lem-3-2}. Denote by $L=\{v\in V(G)\mid d(v)\leq (\frac{1}{2}-8\sqrt{\varepsilon})n\}$. Next we will show that most vertices have degree close to $\frac{n}{2}$ by showing $|L|$ is `very small'. 
\begin{lemma}\label{lem-3-3}
    $|L|\leq \sqrt{\varepsilon}n$
\end{lemma}
\begin{proof}
    Suppose to the contrary that $|L|>\sqrt{\varepsilon}n$. Then there exists a subset $L'\subseteq L$ with $|L'|=\lfloor\sqrt{\varepsilon}n\rfloor$. Combining this with the fact $e(G)\geq e(T_{n,2}) -\varepsilon n^2$, we get
   \begin{align*}
       e(G-L')&\geq e(G)-\sum_{v\in L'}d(v)\\
       &\geq \frac{n^2-1}{4}-\sqrt{\varepsilon}n(\frac{1}{2}-8\sqrt{\varepsilon})n\\
       &>\frac{(1-2\sqrt{\varepsilon}+\varepsilon)n^2+2(1-\sqrt{\varepsilon})n+1}{4}\\
       &=\frac{(n-\sqrt{\varepsilon}n)^2+2(n-\sqrt{\varepsilon} n)+1}{4}\\
       &>\ex(n',C_{2k+1}^{\square}).
   \end{align*}
   where $n'=n-|L'|=n-\lfloor\varepsilon n\rfloor$ and the last inequality follows from Theorem \ref{thm-1-1}. Therefore, $G-L'$ contains a $C_{2k+1}^{\square}$ as a subgraph. This contradicts the fact that $G$ is $C_{2k+1}^{\square}$-free.
\end{proof}

For $i\in \{1,2\}$, let $W_i=\{v\in V_i\mid d_{V_i}(v)\geq 2\sqrt{\varepsilon}n\}$ and $W=W_1\cup W_2$. We derive an upper bound for $|W|$.
\begin{lemma}\label{lem-3-4}
    $|W|\leq \sqrt{\varepsilon}n$
\end{lemma}
\begin{proof}
    By the definition of $W$, we have 
    $$
    e(V_1)+e(V_2)=\sum_{v\in V_1}d_{V_1}(v)+\sum_{v\in V_2}d_{V_2}(v)\geq \sum_{v\in W_1}d_{V_1}(v)+\sum_{v\in W_2}d_{V_2}(v) \geq |W|\sqrt{\varepsilon}n.
    $$
 Since $e(V_1)+e(V_2)\le\varepsilon n^2$, we have 
  $$
    \varepsilon n^2\geq e(V_1)+e(V_2)\geq |W|\sqrt{\varepsilon}n.
    $$
    It leads to $|W|\leq \sqrt{\varepsilon}n$.
\end{proof}

Let $v^*$ be the vertex satisfying $x_{v^*} = \max \{ x_v\mid v \in V(G) \setminus W \}$, and let $u^*$ be the vertex satisfying $x_{u^*} = \max \{ x_v \mid v \in V(G) \}=1$. Next we will prove that $L$ is empty.
\begin{lemma}\label{lem-3-5}
    $L=\emptyset$.
\end{lemma}
\begin{proof}
   First we claim that $v^* \notin L$. Noticing $x_{u^*}=\max\left\{x_v\mid v\in V(G)\right\}=1$, we get
    $$
    \lambda(G)=\lambda(G)x_{u^*}=\sum_{u\in N(u^*)\cap W}x_u+\sum_{v\in N(u^*)\setminus W}x_v< |W|+(n-|W|)x_{v^*}.
    $$
    Combing with Lemma \ref{lem-3-1}, we have 
\begin{equation}\label{ineq1}
    x_{v^*}>\frac{\lambda(G)-|W|}{n-|W|}\ge\frac{\lambda(G)-|W|}{n}\ge \frac{1}{2}-\sqrt{\varepsilon}>\frac{7}{16}.    
\end{equation}
On the other hand, we have 
\[
    \lambda(G) x_{v^*}=\sum_{v\in N(v^*)}x_v=\sum_{v\in N(v^*)\cap W}x_v+\sum_{v\in N(v^*)\setminus W}x_v\le |W|+(d(v^*)-|W|)x_{v^*}.
\]
Therefore, by Eq.\eqref{ineq1}, we get
$$ d(v^*) \ge \lambda(G) +|W|-\frac{|W|}{x_{v^*}}
    \ge \lambda(G) -\frac{9}{7}|W|\geq \frac{n}{2}-\frac{9}{7}\sqrt{\varepsilon}n>\frac{n}{2}-8\sqrt{\varepsilon}n.$$
Thus, $v^*\notin L$, and so $v^*\in V(G)\setminus (W \cup L)$. 

Without loss of generality, we assume that $v^*\in V_1\setminus (W\cup L)$. By the definition of $W$, we have $d_{V_1}(v^*)\leq 2\sqrt{\varepsilon}n$. Then, we obtain that
\begin{align*}
    \lambda(G) x_{v^*}
    &=\sum_{u\in N_{W\cup L}(v^*)}x_u+\sum_{u\in N_{V_1\setminus (W\cup L)}(v^*)}x_u+\sum_{u\in N_{V_2\setminus (W\cup L)}(v^*)}x_u\\
    &<|W|+|L|x_{v^*}+2\sqrt{\varepsilon} nx_{v^*}+\sum_{u\in V_2\setminus (W\cup L)}x_u,
\end{align*}
which implies that 
\begin{equation}\label{ineq2}
    \sum_{u\in V_2\setminus (W\cup L)}x_u\ge (\lambda(G)-|L|-2\sqrt{\varepsilon} n)x_{v^*}-|W|.     
\end{equation}

Suppose to the contrary that there is a vertex $v_0\in L$. Without of loss of generality, assume that $v_0\in V_1\cap L$. Then $d(v_0)\le (\frac{1}{2}-8\sqrt{\varepsilon})n$. Let $G'$ be the graph defined as
\[G'=G-\left\{uv_0\mid uv_0\in E(G)\right\}+\left\{wv_0\mid w\in V_2 \setminus(W\cup L)\right\}.\] 

First, we claim that $G'$ is $C_{2k+1}^{\square}$-free. Otherwise, assume that $V(G')$ contains a subset $H$ which induces a copy of  $C_{2k+1}^{\square}$. Then $v_0\in H$. Assume that $N_H(v_0)=\{v_1,v_2,v_3\}$. For any vertex $v_i\in N_H(v_0)$, since $v_i\notin L$, we have $d(v_i)>(\frac{1}{2}-8\sqrt{\varepsilon})n$. Moreover, since $v_i\notin W$, we have 
    \[
        d_{V_{1}}(v_i)\geq d(v_i)-2\sqrt{\varepsilon}n\geq (\frac{1}{2}-10\sqrt{\varepsilon})n.
    \]
It yields that
\[|\bigcap\limits_{i\in[3]}N_{V_{1}}(v_i)\setminus(W \cup L\cup H)|\ge  \sum_{i=1}^{3}d_{V_1}(v_i)-2|V_1|-|W|-|L|-|H|> (1/2-36\sqrt{\varepsilon})n-2k-1>1.\]
This means that $v_1, v_2, v_{3}$ have another common neighbor $v$ in $V_1$. Therefore, the set $(H\setminus \{v_0\})\cup \{v\}]$ induces a copy of $C_{2k+1}^{\square}$ in $G$, a contradiction.

Next, by Eq.\eqref{ineq1} and Eq.\eqref{ineq2}, we have 
\begin{align*}
    \lambda(G')-\lambda(G)
    &\ge\frac{{\mathbf{x}}^T (A(G')-A(G)){\mathbf{x}}}{{\mathbf{x}^T}{\mathbf{x}}}=\frac{2x_{v_0}}{{\mathbf{x}^T}{\mathbf{x}}}\left( \sum_{u\in V_2\setminus (W\cup L)}x_u-\sum_{uv_0\in E(G)}x_u\right)\\
    &\ge \frac{2x_{v_0}}{{\mathbf{x}^T}{\mathbf{x}}}\left(\left(\left(\lambda(G)-|L|-2\sqrt{\varepsilon}n)x_{v_0}-|W|\right)-(|W|+(d(v_0)-|W|)x_{v^*}\right)\right)\\
    &\ge \frac{2x_{v_0}}{{\mathbf{x}^T}{\mathbf{x}}}((\lambda(G)-|L|-2\sqrt{\varepsilon}n-d(v_0)+|W|)x_{v^*}-2|W|)\\
    &> \frac{2x_{v_0}}{{\mathbf{x}^T}{\mathbf{x}}}(5\sqrt{\varepsilon} nx_{v^*}-2\sqrt{\varepsilon}n)\ge\frac{2x_{v_0}}{{\mathbf{x}^T}{\mathbf{x}}}\left(\frac{3}{16}\sqrt{\varepsilon}n\right)>0.
\end{align*}
This contradicts the maximality of $\lambda(G)$, and thus $L$ must be empty.
\end{proof}

\begin{lemma}\label{lem-3-7}
    For any $v\in V(G)$, we have $(1-92\sqrt{\varepsilon})x_{v^*}\leq x_v$. Furthermore, $x_v\geq \frac{1}{2}-47\sqrt{\varepsilon}$ for any $v\in V(G)$.
\end{lemma}
\begin{proof}
Without loss of generality, assume $ v^* \in V_1 \setminus W $. Thus, we get
\[d(v^*)\leq (\frac{n}{2}+\frac{3\sqrt{\varepsilon}n}{2})+2\sqrt{\varepsilon}n<\frac{n}{2}+4\sqrt{\varepsilon}n.\]

First, consider $ v \in V_1 \setminus W $ with $ v \neq v^* $. From Lemma \ref{lem-3-2}, we have
\[\begin{array}{lll}
    |N(v) \cap N(v^*)|&\geq& |N_{V_2}(v) \cap N_{V_2}(v^*)|
    \geq d_{V_2}(v)+d_{V_2}(v^*)-|V_2|\\[2mm]
    &\geq& 2\left(\frac{1}{2}-10\sqrt{\varepsilon}\right)n-\left(\frac{1}{2}+\frac{3\sqrt{\varepsilon}}{2}\right)n> \frac{n}{2}-22\sqrt{\varepsilon}n,
\end{array}\]
and 
$$N(v^*)\setminus (N(v)\cup W)\leq d(v^*)-|W|-|N(v) \cap N(v^*)|\leq 25\sqrt{\varepsilon}n.$$
Noticing that $x_{v^*}>\frac{1}{2}-\sqrt{\varepsilon}$, $\lambda(G)\geq \frac{1}{2}n$ and $|W|\leq \sqrt{\varepsilon}n$, we have $3x_{v^*}>1$ and $|W|\leq \sqrt{\varepsilon}n\leq 2\sqrt{\varepsilon}\lambda(G)$. Note that
\[\begin{array}{lll}
\lambda(G)(x_v-x_{v^*})&=&\sum_{w\in N(v)\setminus N(v^*)}x_w-\sum_{w\in N(v^*)\setminus N(v)}x_w\ge -\sum_{w\in N(v^*)\setminus N(v)}x_w\\[2mm]
&\ge& -\left(\sum_{w\in W}x_w+\sum_{w\in N(v^*)\setminus (N(v)\cup W)}x_w\right)\\[2mm]
&\ge&-\left(|W|+|N(v^*)\setminus (N(v)\cup W)|x_{v^*}\right)\ge-\left(\sqrt{\varepsilon}n+25\sqrt{\varepsilon}nx_{v^*}\right)\\[2mm]
&\ge&-28\sqrt{\varepsilon}nx_{v^*}\ge -56\sqrt{\varepsilon}\lambda(G)x_{v^*}.
\end{array}\]
This yields $x_v\ge (1-56\sqrt{\varepsilon})x_{v^*}$.

Next, we consider $ v \in V_2 \setminus W $. Since $d_{V_1\setminus W}(v)\geq (\frac{1}{2}-8\sqrt{\varepsilon}-3\sqrt{\varepsilon})n\geq \frac{n}{2}-11\sqrt{\varepsilon}n$, we have
$$\sum_{w\in N_{V_1}(v)\setminus W}x_w\geq d_{V_1\setminus W}(v)(1-56\sqrt{\varepsilon})x_{v^*}> \left(\frac{1}{2}-39\sqrt{\varepsilon}\right)nx_{v^*}.$$
Similarly, we have
\[\begin{array}{lll}
\lambda(G)(x_v-x_{v^*})&=&(\sum_{w\in N(v_1)\cap W}x_w+\sum_{w\in N(v_1)\setminus W}x_w)-(\sum_{w\in N(v^*)\cap W}x_w+\sum_{w\in N(v^*)\setminus W}x_w)\\[2mm]
&\ge& \sum_{w\in N_{V_1}(v)\setminus W}x_w-(|W|+d(v^*)x_v^*)\\[2mm]
&\ge& \left(\frac{1}{2}-39\sqrt{\varepsilon}\right)nx_{v^*}-\sqrt{\varepsilon}n-(\frac{1}{2}+4\sqrt{\varepsilon})nx_{v^*}\\[2mm]
&=&-43\sqrt{\varepsilon}nx_{v^*}-\sqrt{\varepsilon}n\ge -46\sqrt{\varepsilon}nx_{v^*}\ge -92\sqrt{\varepsilon}x_{v^*}\lambda(G).
\end{array}\]
This yields $x_v\ge (1-92\sqrt{\varepsilon})x_{v^*}$ for $v\in V\setminus W$.

Finally, for $v\in W$, we have 
\[\begin{array}{lll}
\lambda(G)&=&\sum_{u\in W\cap N(v)}x_v+\sum_{u\in N_{V_1}\setminus W}x_u+\sum_{u\in N_{V_2}\setminus W}x_u\\[2mm]
&\geq& |N_{V_1}(v)\setminus W|(1-92\sqrt{\varepsilon})x_{v^*}+|N_{V_1}(v)\setminus W|(1-92\sqrt{\varepsilon})x_{v^*}\\[2mm]
& \geq & (d(v)-|W|)(1-92\sqrt{\varepsilon})x_{v^*}\\[2mm]
&\geq & (1-92\sqrt{\varepsilon})x_{v^*}
\end{array}\]
This yields $x_v\ge (1-92\sqrt{\varepsilon})x_{v^*}$ for $v\in V$. 

Since by Eq. \ref{ineq1}, we have $x_v\geq  (1-92\sqrt{\varepsilon})x_{v^*}\geq \frac{1}{2}-47\sqrt{\varepsilon}$.
\end{proof}

 \begin{lemma}\label{lem-3-8}
The graph $G$ contains non of the following graphs as a subgraph:
\begin{itemize}
\vspace{-2mm}
\item[$\operatorname{(i)}$]a cycle $C_4=v_1v_{2k+1}u_{2k+1}u_1v_1$ such that $v_1,v_{2k+1}\in V_i$ and $u_1,u_{2k+1}\in V_j$ for $\{i,j\}=\{1,2\}$;
\vspace{-2mm}
\item[$\operatorname{(ii)}$]a graph $C_6^+$ formed from a cycle $C_6=v_1v_{2k+1}u_{2k}v_{2k}v_0u_1v_1$ by adding the edge $v_{2k+1}v_0$ such that $v_0,v_1,v_{2k},v_{2k+1}\in V_i$ and $u_1,u_{2k}\in V_j$ for $\{i,j\}=\{1,2\}$.
\end{itemize}
\end{lemma}
\begin{proof}
We show that, if either (i) or (ii) holds, then $G$ will contains a copy of $ C_{2k+1}^{\square}$ as a subgraph, which leads to a contradiction. Without loss of generality, assume $i=1$ and $j=2$. 

First, we show that, if there exist $v_1\in V_1$ and $u_1\in V_2$ with $v_1\sim u_1$, then $G$ contains $P_{2k-1}^\square$ with $V=\{v_1,\ldots,v_{2k-1},u_1,\ldots,u_{2k-1}\}$ such that $v_1u_2v_3u_4\cdots u_{2k-2}v_{2k-1}$ and $u_1v_2u_3v_4\cdots v_{2k-2}u_{2k-1}$ form two copies of $P_{2k-1}$ and $v_iu_i\in E(G)$ for $1\le i\le 2k-1$. 

Since $d(v_1)\ge \left(\frac{1}{2}-8\sqrt{\varepsilon}\right)n$, we get $d_{V_2}(v_1)\ge \left(\frac{1}{4}-4\sqrt{\varepsilon}\right)n$ due to the maximality of the cross edges between $V_1$ and $V_2$. There exists $u_2\in V_2\setminus W$ such that $u_2\sim v_1$. This leads to $d_{V_1}(u_2)\ge \left(\frac{1}{2}-10\sqrt{\varepsilon}\right)n$. Therefore, since 
\[\begin{array}{lll}|N_{V_1}(u_1)\cap N_{V_1}(u_2)|&\ge& d_{V_1}(u_1)+d_{V_1}(u_2)-|V_1|\\[2mm]
&\ge& \left(\frac{1}{4}-4\sqrt{\varepsilon}\right)n+\left(\frac{1}{2}-10\sqrt{\varepsilon}\right)n-\left(\frac{1}{2}+\frac{3\sqrt{\varepsilon}}{2}\right)n\\[2mm]
&>&\left(\frac{1}{4}-16\sqrt{\varepsilon}\right)n,
\end{array}\]
there exists $v_2\in V_1\setminus W$ such that $v_2\in N(u_1)\cap N(u_2)$. 

Now, it remains to find $v_{i+1}$ and $u_{i+1}$, given that $v_j\in V_1\setminus W$ and $u_j\setminus W$ have already been obtained for all $2\le j\le i$. Since $v_i\not\in W$, we get $d_{V_2}(v_i)\ge\left(\frac{1}{2}-10\sqrt{\varepsilon}\right)n$. Therefore, there exists $u_{i+1}V_2\setminus(W\cup\{u_1,u_2,\ldots,u_{i}\})$ such that $u_{i+1}\sim v_i$. Also, by noticing that
\[\begin{array}{lll}|N_{V_1}(u_{i})\cap N_{V_1}(u_{i+1})|&\ge& d_{V_1}(u_i)+d_{V_1}(u_{i+1})-|V_1|\\[2mm]
&\ge& \left(\frac{1}{2}-10\sqrt{\varepsilon}\right)n+\left(\frac{1}{2}-10\sqrt{\varepsilon}\right)n-\left(\frac{1}{2}+\frac{3}{2}\sqrt{\varepsilon}\right)n\\[2mm]
&>&\left(\frac{1}{2}-22\sqrt{\varepsilon}\right)n,\end{array}\]
there exists $v_{i+1}\in V_1\setminus(W\cup\{v_1,v_2,\ldots,v_{i}\})$ such that $v_{i+1}\in N(u_i)\cap N(u_{i+1})$. Hence, the desired $P_{2k-1}^{\square}$ is obtained.

Next, assume that (i) holds. By noticing that 
\[|N_{V_2}(v_{2k-1})\cap N_{V_2}(v_{2k+1})|\ge \left(\frac{1}{2}-10\sqrt{\varepsilon}\right)n+\left(\frac{1}{4}-4\sqrt{\varepsilon}\right)n-\left(\frac{1}{2}+\frac{3}{2}\sqrt{\varepsilon}\right)n>\left(\frac{1}{4}-16\sqrt{\varepsilon}\right)n,\]
there exists $u_{2k}\in V_2\setminus (W\cup \{u_1,\ldots,u_{2k-1},u_{2k+1}\})$ such that $u_{2k}\in N(v_{2k-1})\cap N(v_{2k+1})$. Also, since
\[\begin{array}{lll}&&|N_{V_1}(u_{2k-1})\cap N_{V_1}(u_{2k})\cap N_{V_1}(u_{2k+1})|\\[2mm]
&\ge& \left(\frac{1}{2}-10\sqrt{\varepsilon}\right)n+\left(\frac{1}{2}-10\sqrt{\varepsilon}\right)n+\left(\frac{1}{4}-4\sqrt{\varepsilon}\right)n-2\left(\frac{1}{2}+\frac{3}{2}\sqrt{\varepsilon}\right)n\\[2mm]
&=&\left(\frac{1}{4}-27\sqrt{\varepsilon}\right)n,\end{array}\]
there exists $v_{2k}\in V_1\setminus (W\cup\{v_1,\ldots,v_{2k-1},v_{2k+1}\})$ such that $v_{2k}\in N_{V_1}(u_{2k-1})\cap N_{V_1}(u_{2k})\cap N_{V_1}(u_{2k+1})$. Therefore, the vertices $\{v_1,\ldots,v_{2k+1},u_1,\ldots,u_{2k+1}\}$ form a desired $C_{2k+1}^{\square}$ (see Fig.\ref{fig-3-1}).

Thirdly, assume that (ii) holds. By noticing that 
\[|N_{V_2}(v_{2k-2})\cap N_{V_2}(v_{2k})|\ge \left(\frac{1}{2}-10\sqrt{\varepsilon}\right)n+\left(\frac{1}{4}-4\sqrt{\varepsilon}\right)n-\left(\frac{1}{2}+\frac{3}{2}\sqrt{\varepsilon}\right)n>\left(\frac{1}{4}-16\sqrt{\varepsilon}\right)n,\]
there exists $u'_{2k-1}\in V_2\setminus (W\cup \{u_1,\ldots,u_{2k-2},u_{2k}\})$ such that $u'_{2k-1}\in N(v_{2k-2})\cap N(v_{2k})$. Also, since
\[\begin{array}{lll}&&|N_{V_1}(u_{2k-2})\cap N_{V_1}(u_{2k-1})\cap N_{V_1}(u_{2k})|\\[2mm]
&\ge& \left(\frac{1}{2}-10\sqrt{\varepsilon}\right)n+\left(\frac{1}{2}-10\sqrt{\varepsilon}\right)n+\left(\frac{1}{4}-4\sqrt{\varepsilon}\right)n-2\left(\frac{1}{2}+\frac{3}{2}\sqrt{\varepsilon}\right)n\\[2mm]
&=&\left(\frac{1}{4}-27\sqrt{\varepsilon}\right)n,\end{array}\]
there exists $v'_{2k-1}\in V_1\setminus (W\cup\{v_0,v_1,\ldots,v_{2k-2},v_{2k},v_{2k+1}\})$ such that $v'_{2k-1}\in N_{V_1}(u_{2k-2})\cap N_{V_1}(u_{2k-1})\cap N_{V_1}(u_{2k})$. Therefore, the vertices \[\{v_0,v_1,\ldots,v_{2k-2},v'_{2k-1},v_{2k},v_{2k+1},u_1,\ldots,u_{2k-2},u'_{2k-1},u_{2k}\}\] form a desired $C_{2k+1}^{\square}$ (see Fig.\ref{fig-3-1}).
\end{proof}
\begin{figure}[htbp]
\begin{center}
\includegraphics[width=360pt]{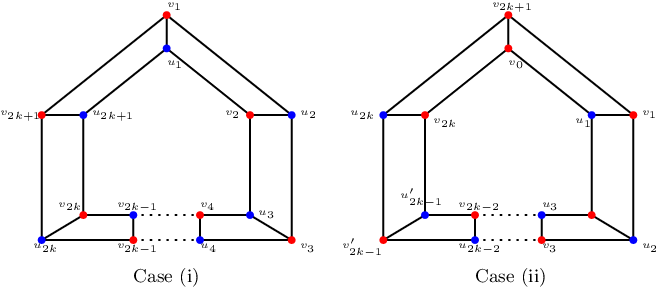}
\caption{\small Graphs used in the proof of Lemma \ref{lem-3-8}, where the vertices in $V_1$ are colored red and those in $V_2$ are colored blue.}
\label{fig-3-1}
\end{center}
\end{figure}

\begin{cor}\label{cor-3-9}
For any $i\in\{1,2\}$, the is no path $P_4=v_1v_2v_3v_4$ in $V_i$ with $v_1,v_4\not\in W$.
\end{cor}
\begin{proof}
Suppose to the contrary that such path $P_4$ exists. For $j\ne i$, we have
\[|N_{V_j}(v_1)\cap N_{V_j}(v_3)|\ge \left(\frac{1}{2}-10\sqrt{\varepsilon}\right)n+\left(\frac{1}{4}-4\sqrt{\varepsilon}\right)n-\left(\frac{1}{2}+\frac{3}{2}\sqrt{\varepsilon}\right)n>\left(\frac{1}{4}-16\sqrt{\varepsilon}\right)n.\]
This means there exists $u_1\in N_{V_j}(v_1)\cap N_{V_j}(v_3)$. Similarly, there exists $u_2\in N_{V_j}(v_2)\cap N_{V_j}(v_4)$. Thus, $\{v_1,\ldots,v_4,u_1,u_2\}$ form a $C_6^+$ defined in (ii) of Lemma \ref{lem-3-8}, a contradiction. 
\end{proof}

\begin{cor}\label{cor-3-10}
    For any $i\in \{1,2\}$, we get the following statements:
    \begin{itemize}
        \item[(i)] $W_i$ is an independent set;
        \item[(ii)] If $v_1v_2$ is an edge in $V_i\setminus W_i$, then $v_1,v_2\in V_i\setminus(W_i\cup N_{V_i}(W_i))$;
        \item[(iii)] Each non-trivial component in $V_i\setminus(W_i\cup N_{V_i}(W_i))$ is either a triangle or a star.
    \end{itemize}
\end{cor} 
\begin{proof}
If there is an edge $v_2v_3$ in $W_i$, then there is a path $v_1v_2v_3v_4$ such that $v_1,v_4\in V_i\setminus W_i$ because $N_{V_i}(v_t)\ge 2\sqrt{\varepsilon}n>|W|+3$ for $t\in\{2,3\}$, which contradicts Corollary \ref{cor-3-9}. Hence, (i) holds.
If there is an edge $v_1v_2$ in $V_i\setminus W_i$ with $v_2\in N_{V_i}(W)$, then let $v_3\in W_i$ with $v_3\sim v_2$. Since $N_{V_i}(v_3)\ge 2\sqrt{\varepsilon}n>|W|+3$, there exists $v_4\in V_i\setminus (W\cup \{v_1,v_2\})$ such that $v_3\sim v_4$, which contradicts Corollary \ref{cor-3-9}. Hence, (ii) holds. 

For convenience, denote by $\overline{V_i}=V_i\setminus(W_i\cup N_{V_i}(W_i))$. Let $H$ be a non-trivial component of $G[\overline{V_i}]$. If $H$ contains a triangle $C_3=v_1v_2v_3v_1$, then $H=C_3$ since otherwise there would be a vertex $v_4\in V(H)$ such that $v_4\sim v_i$ for some $i\in \{1,2,3\}$. This yield a $P_4=v_1v_2v_3v_4$ with $v_1,v_4\not\in W$, which contradicts Corollary \ref{cor-3-9}. If $H$ is $C_3$-free, since $H$ is also $P_4$-free due to Corollary \ref{cor-3-9}, then $H$ is a star. Hence, (iii) holds. 
\end{proof}

\begin{lemma}\label{lem-3-11}
For any $i\in\{1,2\}$, if $w_1,w_2\in W_i$ share a common neighbor in $V_i$, then $N_{V_j}(w_1)\cap N_{V_j}(w_2)=\emptyset$ where $j\ne i$.
\end{lemma}
\begin{proof}
Suppose to the contrary that $u\in N_{V_j}(w_1)\cap N_{V_j}(w_2)$. Since $d_{V_i}(w_1)\ge 2\sqrt{\varepsilon}n>2$, there exists $v'\in N_{V_i}(w_1)\setminus \cup\{v\}$. By noticing that
\[|N_{V_j}(v)\cap N_{V_j}(v')|\ge\left( \frac{1}{2} - 10\sqrt{\varepsilon} \right)n+\left( \frac{1}{2} - 10\sqrt{\varepsilon} \right)n-\left( \frac{1}{2} +\frac{3}{2}\sqrt{\varepsilon} \right)n>\left( \frac{1}{2} - 22\sqrt{\varepsilon} \right)n,\]
there exists $u'\in N_{V_j}(v)\cap N_{V_j}(v')$. Thus, $\{w_1,w_2,v,v',u,u'\}$ form a $C_6^+$ defined in Lemma \ref{lem-3-8}, a contradiction.
\end{proof}

For $i\in\{1,2\}$, let $W_i^*=\{w\in W_i\mid \exists w'\in W_i\text{ such that }N_{V_i}(w)\cap N_{V_i}(w')\ne\emptyset\}$, and $W^*=W_1^*\cup W_2^*$.
\begin{lemma}\label{lem-3-12}
 For any $i\in\{1,2\}$, if $w\in W_i^*$, then 
 \[\left(\frac{1}{4}-14\sqrt{\varepsilon}\right) n\leq d_{V_i}(w)\leq d_{V_j}(w)\leq \left(\frac{1}{4}+6\sqrt{\varepsilon}\right) n,\] 
where $j\ne i$. Furthermore, $|W^*|\leq 8$.
\end{lemma}
\begin{proof}
Suppose to the contrary that there exists $w\in W_i^*$ such that $d_{V_j}(w)\ge\left(\frac{1}{4}+6\sqrt{\varepsilon}\right) n$. Let $w'\in W_i$ such that $N_{V_i}(w)\cap N_{V_i}(w')\ne\emptyset$. By noticing that
\[|N_{V_j}(w)\cap N_{V_j}(w')|\ge \left(\frac{1}{4}+6\sqrt{\varepsilon}\right) n+\left(\frac{1}{4}-4\sqrt{\varepsilon}\right) n-\left(\frac{1}{2}+\frac{3}{2}\sqrt{\varepsilon}\right) n>0,\]
we get $N_{V_j}(w)\cap N_{V_j}(w')\ne\emptyset$, which contradicts Lemma \ref{lem-3-11}. Hence, $d_{V_j}(w)< \left(\frac{1}{4}+6\sqrt{\varepsilon}\right) n$. This in turn leads to 
\[d_{V_i}(w)\ge \left(\frac{1}{2}-8\sqrt{\varepsilon}\right) n-\left(\frac{1}{4}+6\sqrt{\varepsilon}\right) n=\left(\frac{1}{4}-14\sqrt{\varepsilon}\right) n.\]

Next, for any $v\in V_i$, we claim that $d_{W_i^*}(v)\le 2$. Otherwise, assume $w_1,w_2,w_3\in N_{W_i^*}(v)$. By the arguments above, we have $d_{V_j}(w_t)\ge \left(\frac{1}{4}-14\sqrt{\varepsilon}\right) n$. This means $\sum_{t=1}^3d_{V_j}(w_t)>n$. Therefore, there must exists $\{t_1,t_2\}\subset\{1,2,3\}$ such that $N_{V_j}(w_{t_1})\cap N_{V_j}(w_{t_2})\ne\emptyset$, which contradicts Lemma \ref{lem-3-11}. Counting the number of edges between $W_i^*$ and $N_{V_i}(W_i^*)$, we get
\[\left(\frac{1}{4}-14\sqrt{\varepsilon}\right) n|W_i^*|\le e(W_i^*,N_{V_i}(W_i^*))\le 2|N_{V_i}(W_i^*)|\le 2|V_i|.\]
Therefore, we have
\[\left(\frac{1}{4}-14\sqrt{\varepsilon}\right) n|W^*|\le e(W_1^*,N_{V_1}(W_1^*))+e(W_2^*,N_{V_2}(W_2^*))\le 2(|N_{V_1}(W_1^*)|+|N_{V_2}(W_2^*)|)\le 2n.\]
This yields $|W^*|\le 8$.
\end{proof}

For $i=1,2$, let $\mathcal{X}^{(i,\triangle)}$ (resp. $\mathcal{X}^{(i,*)}$)  be the collection of components in $G[V_i]$ that are isomorphic to triangles (resp. stars), and let $\mathcal{X}^{(i,o)}$ denote the set of all remaining non-trivial components in $G[V_i]$.
Denote by $\mathcal{X}^{(i)}=\mathcal{X}^{(i,\triangle)}\cup \mathcal{X}^{(i,*)}\cup \mathcal{X}^{(i,o)}$.
\begin{lemma}\label{lem-3-13}
For $1\le i\le 2$, let $A_i\subseteq V_i$ be such that $G[A_i]\in\mathcal{X}^{(i,*)}$ and let $B_i\subseteq V_i$ be such that $G[B_i]\in \mathcal{X}^{(i,\triangle)}$. If $|A_i|=a_i$ and $|B_i|=b_i$, then $e(A_1,A_2)\leq a_1a_2-\min\{a_1,a_2\}$, $e(A_1,B_2)\leq 3a_1-3$, and $e(B_1,B_2)\leq 5$.
\end{lemma}
\begin{proof}
Without loss of generality, assume $a_1\le a_2$. Let $A_i=\{v_1^{(i)},v_2^{(i)},\ldots,v_{a_i}^{(i)}\}$ for $1\le i\le 2$ where $v_1^{(i)}$ is the center of $G[A_i]$. 

Suppose to the contrary that $e(A_1,A_2)> a_1a_2-a_1$. Then $N_{V_1}(v_t^{(2)})\ne \emptyset$ for any $1\le t\le a_2$. If $v_1^{(1)}\sim v_1^{(2)}$, then $v_r^{(1)}\not\sim v_t^{(2)}$ for any $2\le r\le a_1$ and $2\le t\le a_2$ since otherwise there would exist a $C_4$ as defined in Lemma \ref{lem-3-8}. Therefore, $N_{V_1}(v_t^{(2)})=\{v_1^{(1)}\}$ for $2\le t\le a_2$. Moreover, if there is a vertex $v_r^{(1)}$ adjacent to $v_1^{(2)}$, then $\{v_1^{(1)},v_r^{(1)},v_1^{(2)},v_2^{(2)}\}$ also forms a $C_4$ as defined in Lemma \ref{lem-3-8}. Thus, $N_{V_2}(v_r^{(1)})=\emptyset$ for $2\le r\le a_1$. Hence, by $a_2\ge a_1\ge 2$, we get $e(A_1,A_2)=a_2<a_1a_2-a_1$, a contradiction. If $v_1^{(1)}\not\sim v_1^{(2)}$, then we may assume $v_1^{(2)}\sim v_2^{(1)}$. Note that, $v_1^{(1)}\not\sim v_t^{(2)}$ for any $2\le t\le a_2$ since otherwise $\{v_1^{(1)},v_2^{(1)},v_1^{(2)},v_t^{(2)}\}$ would form a $C_4$ as defined in Lemma \ref{lem-3-8}. Therefore, $N_{V_2}(v_1^{(1)})=\emptyset$. This indicates $e(A_1,A_2)\le a_1a_2-a_2\le a_1a_2-a_1$, a contradiction. 

 Since $B_2$ contains a star on $3$ vertices, by the arguments above, we get $e(A_1,B_2)\le 3a_1-\min\{a_1,3\}$. If $a_1\ge 3$ then we are done. If $a_1=2$, one could easily obtained a $C_4$ as defined in Lemma \ref{lem-3-8} whenever $e(A_1,B_2)>3$. One could easily get $e(B_1,B_2)\le 5$ similarly.
\end{proof}

\begin{lemma}\label{lem-3-14}
    $\mathcal{X}^{(1,o)}\cup \mathcal{X}^{(2,o)} =\emptyset$.
\end{lemma}
\begin{proof}
Suppose to the contrary that $\mathcal{X}^{(1,o)}\cup \mathcal{X}^{(2,o)}\ne\emptyset$, say $H$ induces a component in $\mathcal{X}^{(1,o)}$.  From Corollary \ref{cor-3-10}, we get $H\subseteq W_1\cup N_{V_1}(W_1)$. Let $W_H=H\cap W_1$ and $\overline{W_H}=H\cap N_{V_1}(W_1)$. Clearly, $W_H\subseteq W_1^*$ and $|W_H|\ge 2$. For a given vertex $w_0\in W_H$, we define the graph $G'$ with vertex set $V(G)$ and edge set 
\[E(G')=\{v_1v_2\mid v_i\in V_i\text{ for }i\in\{1,2\}\}\cup E(G(V_1\setminus W_1^*))\cup \{w_0v\mid v\in N(w_0)\}.\]
Since $G'[V_1]$ is a union of some stars an triangles and $G'[V_1]$ is an empty graph, when we the vertices in $V_1$ red and those in $V_2$ blue, there is either monochromatic path $P_4$ nor a cycle $C_4=uvwz$ with $u$ and $v$ are red while $w$ and $z$ are blue. Therefore, Corollary \ref{cor-2-5} indicates that $G'$ is also $C_{2k+1}^{\square}$-free. We derive a contradiction by showing $\lambda(G') > \lambda(G)$, which is established by the inequality $\mathbf{x}^T(A(G')-A(G))\mathbf{x} > 0$.

For convenience, all symbols with a prime superscript denote the corresponding objects in $G'$. For instance, $N'(v)$ represents the neighborhood of $v$ in $G'$, $d'(v)$ denotes the degree of $v$ in $G'$, and so on. Note that
\begin{equation}\label{eq-f-1}
\begin{array}{lll}
&&(\mathbf{x}^T(A(G')-A(G))\mathbf{x})/2=\left(\sum_{uv\in E(G')}x_ux_v-\sum_{uv\in E(G)}x_ux_v\right)\\[2mm]
&\ge&\sum_{v\in V_2\setminus N_{V_2}(w_0)}x_{w_0}x_v+\sum_{w\in W^*\setminus\{w_0\}}x_w\left(\sum_{v\in N'(w)}x_v-\sum_{v\in N(w)}x_v\right)\\[2mm]
&&+\left(\sum_{v\in V_2\setminus W_2^*}x_v\left(\sum_{u\in V_1\setminus N_{V_1}(v)}x_u-\sum_{u\in N_{V_2}(v)}x_u\right)-\sum_{uv\in E(G[V_2\setminus W_2^*])}x_ux_v\right).
\end{array}
\end{equation}
From Lemmas \ref{lem-3-7} and \ref{lem-3-12}, the first term in the right hand side of Eq.\eqref{eq-f-1} is lower bounded by
\begin{equation}\label{eq-f-2}
\sum_{v\in V_2\setminus N_{V_2}(w_0)}x_{w_0}x_v>\left(\left(\frac{1}{2}-\frac{3}{2}\sqrt{\varepsilon}\right)-\left(\frac{1}{4}+6\sqrt{\varepsilon}\right)\right)n(\frac{1}{2}-7\sqrt{\varepsilon})^2>\frac{1}{17}n.
\end{equation}
From Lemmas \ref{lem-3-7} and \ref{lem-3-12}, the second term is lower bounded by 
\begin{equation}\label{eq-f-3}
\begin{array}{lll}
&&\sum_{w\in W^*\setminus\{w_0\}}x_w\left(\sum_{v\in N'(w)}x_v-\sum_{v\in N(w)}x_v\right)\\[2mm]
&\ge&\sum_{w\in W^*\setminus\{w_0\}}x_w\left(\left(\left(\frac{1}{2}-\frac{3}{2}\sqrt{\varepsilon}\right)-\left(\frac{1}{4}+6\sqrt{\varepsilon}\right)\right)n(1-92\sqrt{\varepsilon})x_{v^*}-\left(\frac{1}{4}+6\sqrt{\varepsilon}\right)nx_{v^*}\right)\\[2mm]
&\ge&\sum_{w\in W^*\setminus\{w_0\}}(-37\sqrt{\varepsilon})nx_wx_{v^*}\ge-259\sqrt{\varepsilon}n.
\end{array}
\end{equation}
In what follows, we consider the third term in the right hand side of Eq.\eqref{eq-f-1}.

Define the set $T=\{v\in V_2\mid N(v)\cap N_{V_2}(w)\ne\emptyset~\text{for some} ~w\in W_H\}$. We claim, for any $v\in T$, that
\begin{equation}\label{eq-a-1}\frac{n}{4} +16\sqrt{\varepsilon} n\geq d_{V_1}(v)\geq d_{V_2}(v) \geq \frac{n}{4} - 24\sqrt{\varepsilon} n.\end{equation}
Suppose to the contrary that $d_{V_1}(v)>\left(\frac{1}{4} +16\sqrt{\varepsilon}\right) n$. Assume that $v\sim v_1$ where $v_1\in N_{V_2}(w)$ for some $w\in W_H$. According to Lemma \ref{lem-3-12}, we get 
\[|N_{V_1}(w)\cap N_{V_1}(v)|\ge \left(\frac{1}{4}-14\sqrt{\varepsilon}\right)n+\left(\frac{1}{4} +16\sqrt{\varepsilon}\right) n-\left(\frac{1}{2} +\frac{3}{2}\sqrt{\varepsilon}\right) n>0.\]
This means that there exists $v_2\in N_{V_1}(w)\cap N_{V_1}(v)$. Therefore, $\{w,v_1,v_2,v\}$ form a $C_4$ as defined in Lemma \ref{lem-3-8}. Thus, $d_{V_1}(v)\le \left(\frac{1}{4} +16\sqrt{\varepsilon}\right) n$, which in turn leads to 
\[d_{V_2}(v)\ge \left(\frac{1}{2} -8\sqrt{\varepsilon}\right) n-\left(\frac{1}{4} +16\sqrt{\varepsilon}\right) n=\left(\frac{1}{4} -24\sqrt{\varepsilon}\right) n.\]
This means $T\subseteq W_2$. For $w_1,w_2\in W_H$ satisfying that $w_1$ and $w_2$ sharing a common neighbor in $V_1$, we have
\[(N_{V_2}(w_1)\cup N_{V_2}(w_2))\ge 2\left(\frac{1}{4}-14\sqrt{\varepsilon}\right) n=\left(\frac{1}{2}-28\sqrt{\varepsilon}\right) n.\]
Therefore, we get
\begin{equation}\label{eq-f-4}|V_2\setminus T|\le \left(\frac{1}{2}+\frac{3}{2}\sqrt{\varepsilon}\right) n-\left(\frac{1}{2}-28\sqrt{\varepsilon}\right) n\le 30\sqrt{\varepsilon}n,\end{equation}
and 
\[|N_{V_2}(w)\cap(N_{V_2}(w_1)\cup N_{V_2}(w_2))|\ge 3\left(\frac{1}{4}-14\sqrt{\varepsilon}\right) n-\left(\frac{1}{2}+\frac{3}{2}\sqrt{\varepsilon}\right) n>0 \]
for any $w\in W_2^*$. This implies $W_2^*\subseteq T$. Moreover, we have 
\begin{equation}\label{eq-a-0}|T\setminus W_2^*|\le 2.
\end{equation} Otherwise, there exist $v_1,v_2,v_3\in T\setminus W_2^*$. Since $d_{V_2}(v_i)\ge \frac{n}{4} - 24\sqrt{\varepsilon} n$, we get $\sum_{i=1}^3d_{V_2}(v_i)>|V_2|$. It implies that at least two of $v_1,v_2,v_3$ belong to $W_2^*$, a contradiction. Therefore, combining Eqs.\eqref{eq-a-1}, \eqref{eq-f-4}, and \eqref{eq-a-0}, we get
\begin{equation}\label{eq-f-5}\begin{array}{lll}
&&\left(\sum_{v\in V_2\setminus W_2^*}x_v\left(\sum_{u\in V_1\setminus N_{V_1}(v)}x_u-\sum_{u\in N_{V_2}(v)}x_u\right)-\sum_{uv\in E(G[V_2\setminus W_2^*])}x_ux_v\right)\\[2mm]
&\ge &\sum_{v\in T\setminus W_2^*}x_v\left(\sum_{u\in V_1\setminus N_{V_1}(v)}x_u-\sum_{u\in N_{V_2}(v)}x_u\right)-\sum_{uv\in E(V_2\setminus T)x_ux_v}\\[2mm]
&\ge&\sum_{v\in T\setminus W_2^*}x_v\left(\left((\frac{1}{2}-\frac{3\sqrt{\varepsilon}}{2})-(\frac{1}{4}+16\sqrt{\varepsilon})\right)n(1-92\sqrt{\varepsilon})x_{v^*}-(\frac{n}{4}+16\sqrt{\varepsilon}n)x_{v^*}\right)-30\sqrt{\varepsilon}n\\[2mm]
&>&\sum_{v\in T\setminus W_2^*}(-57\sqrt{\varepsilon})nx_vx_{v^*}-30\sqrt{\varepsilon}n\ge -144\sqrt{\varepsilon}n.`
\end{array}\end{equation}
By combining Eqs.\eqref{eq-f-1}, \eqref{eq-f-2}, \eqref{eq-f-3}, and \eqref{eq-f-5}, we get
\[(\mathbf{x}^T(A(G')-A(G))\mathbf{x})/2>\frac{1}{17}n-259\sqrt{\varepsilon}n-144\sqrt{\varepsilon}n>0.\]

The proof is completed.
\end{proof}

Now, we are ready to show our main result.
\begin{proof}[{\bf{Proof of Theorem \ref{thm-1-2}}}]
It only needs to prove that $ G = K_1 \vee T_{n-1,2} $. It follows from Lemma \ref{lem-3-14} that all connected components of $ G[V_1] $ and $ G[V_2] $ are either triangles or stars. We will proceed with the proof by considering the following two cases.

{\flushleft\bf Case 1.}  $V_1$ or $V_2$ is an independent set.

Without loss of generality, let $V_2$ be independent. First, we claim that each vertex in $V_1$ is adjacent to every vertex in $V_2$. Otherwise, we construct $G_1$ from $G$ by adding a missing edge between $V_1$ and $V_2$. Then $G_1$ is also $C_{2k+1}^{\square}$-free by Corollary \ref{cor-2-5} but $\lambda(G_1) > \lambda(G)$, a contradiction. 

Next, we claim that there is only one star in $G[V_1]$ as a component, where an isolated vertex is viewed as a star on one vertex. Otherwise, assume that there are two stars in $G[V_1]$ with centers $v_1$ and $v_2$ satisfying $x_{v_1}\ge x_{v_2}$. Define the graph $G_2$ as
\[G_2=G-\{uv_2\mid u\in N(v_2)\}\cup\{uv_1\mid u\in N(v_2)\cup\{v_2\}\}.\] 
 Then $\lambda(G_2) > \lambda(G)$ and $G_2$ is also $C_{2k+1}^{\square}$-free, a contradiction. 
 
Assume that $G[V_1]=rC_3\cup K_{1,s}$ and $n_2=|V_2|$, where $3r+s+1+n_2=n$. Therefore, $G=(rC_3\cup K_{1,s})\vee n_2K_1$.  Thirdly, we claim that $r=0$. According Lemma \ref{lem-2-2}, the spectral radius $\lambda(G)$ is just the largest root of 
  \[f_{r}(x)=x^4 -2x^3-(n_2(s+3r+1)+s)x^2+2(n_2+s)x+3n_2sr+4n_2s.\]
If $r\ge 1$, then 
\[f_r(x)-f_0(x)=3r(x^2-2x+n_2(n-n_2-3r-5))\]
which is positive whenever $x\ge\lambda(G)\ge n/2$. It implies that $\lambda(K_{1,n-n_2-1}\vee n_2K_1)>\lambda(G)$, a contradiction. Thus, we get $r=0$, that is $G=K_{1,n-n_2-1}\vee n_2K_1=K_1\vee K_{n_1,n_2}$, where $n_1=n-n_1-1$. Furthermore, the spectral radius $\lambda(G)$ is the largest root of 
\begin{equation}\label{eq-g-1}f_0(x)=(x-2)(x^3-(n_2(n-n_2)+(n-n_2-1))x-2).\end{equation}

Finally, we claim $|n_1-n_2|\le 1$. By the symmetry of $n_1$ and $n_2$, we may assume that $n_2\le (n-1)/2$. If $n_1-n_2\ge 2$, let $G_3=K_1\vee K_{n_1-1,n_2+1}$. According to Eq.\eqref{eq-g-1}, the values $\lambda(G)$ and $\lambda(G_3)$ are respectively the largest root of $g(x)$ and $g_3(x)$, where
\[\left\{\begin{array}{l}
g(x)=x^3-(n_2(n_1+1)+n_1)x-2,\\[2mm]
g_3(x)=x^3-((n_2+1)n_1+n_1-1)x-2.
\end{array}\right.\]
Since $g(x)-g_3(x)=(n_1-n_2-1)x>0$ whenever $x>0$, we get $\lambda(G_3)>\lambda(G)$, a contradiction. Hence, $G=K_1\vee T_{n-1,2}$.

{\flushleft\bf Case 2.} Neither $V_1$ nor $V_2$ is an independent set.

Recall that $x_{u^*}=\max\{x_v\mid v\in V(G)\}=1$. 
{\flushleft\bf Subcase 2.1} $x_{u^*} \geq (3/2+50\sqrt{\varepsilon})x_{v^*}$.

Without loss of generality, we may assume $u^* \in V_1$. First, we claim that $d_{V_1}(u^*) \geq \frac{n}{4} + 20\sqrt{\varepsilon} n$, and thus $u^*\in W$. Suppose on the contrary that $d_{V_1}(u^*) < \frac{n}{4} + 20\sqrt{\varepsilon} n$. Since $3x_{v^*}\ge 3(1/2-47\sqrt{\varepsilon})>1$, by Lemma \ref{lem-3-4}, we get
\[\begin{array}{lll}
\lambda(G)x_{u^*}&=&\sum_{u\in N(u^*)\cap W}x_u+\sum_{u\in N(u^*)\setminus W}x_u\\[2mm]
&\le& |W|+d(u^*)x_{v^*}\le |W|+(d_{V_1}(u^*)+|V_2|)x_{v^*}\\[2mm]
&<& 3|W|x_{v^*}+(\frac{3n}{4}+22\sqrt{\varepsilon}n) x_{v^*}\\[2mm]
&\le&  (\frac{3n}{4}+25\sqrt{\varepsilon}n) x_{v^*}.
\end{array}\] 
By noticing $\lambda(G)\ge n/2$, we get $x_{u^*} <(\frac{3}{2}+50\sqrt{\varepsilon})x_{v^*}$, contradicting the assumption. 

Let $H$ denote the component containing $u^*$. Since $u^*\in W$, $H$ is a star. Next, we claim that, for any $v\in V_2$, there is no edge between $v$ and $N_{V_2}(u^*)$. Otherwise, assume $v\in V_2$ and $u\in N_{V_2}(u^*)$ with $u\sim v$. Note that $d_{V_1}(v)\ge \frac{n}{4} -4\sqrt{\varepsilon} n$ due to Lemma \ref{lem-3-5} and the maximality of $e(V_1,V_2)$. Since $d_{V_1}(u^*)\geq \frac{n}{4} +20 \sqrt{\varepsilon} n$, we have 
\[|N_{V_1}(u^*)\cap N_{V_1}(v)|\geq d_{V_1}(u^*)+d_{V_1}(v)-(\frac{n}{2}+\frac{3\sqrt{\varepsilon}n}{2})>0.\] 
Therefore, there exists $u'\in N_{V_1}(u^*)\cap N_{V_1}(v)$. Thus, $G$ contains a $C_4=u^*u'vuu^*$ with $u^*,u'\in V_1$ and $v,u\in V_2$, which contradicts Lemma \ref{lem-3-8}. 

Let $G_4 = G -E(G[V_1\setminus V(H)])-E(G[V_2]) + \{u^*v \mid v \in V(G)\setminus N(u^*)\}$. Note that, Corollary \ref{cor-2-5} indicates that $G_4$ is $C_{2k+1}^{\square}$-free, and Lemma \ref{lem-2-3} implies $\lambda(G_4) >\lambda(G)$. This contradicts the maximality of $\lambda(G)$.

{\flushleft\bf Subcase 2.2} $x_{u^*} \leq (3/2+50\sqrt{\varepsilon})x_{v^*}$.

Let $Q_{min} $ be a non-trivial component with minimum edges. Without loss of generality, assume $Q_{min} \in V_1$. First, we assert that $|\mathcal{X}^{(2)}| = 1$. Otherwise, assume there exist two non-trivial components $A, B \in \mathcal{X}^{(2)}$. By Lemma \ref{lem-3-13} we get $e(A, Q_{min}) \leq |A||Q_{min}|-|Q_{min}|$ and $e(B, Q_{min}) \leq |B||Q_{min}|-|Q_{min}|$. Let 
\[G_5 = G - E(Q_{min}) + E(V(Q_{min}),V(A)\cup V(B)).\]
 It is obvious that $G_5$ is still $C_{2k+1}^{\square}$-free by Corollary \ref{cor-2-5}. Moreover, since at most one vertex of $Q_{min}$ belongs to $W$, we get 
\[
\mathbf{x}^TA(G_5)\mathbf{x}-\mathbf{x}^TA(G)\mathbf{x}\geq 2\left(2|Q_{min}|(1-92\sqrt{\varepsilon})^2x^2_{v^*}-|Q_{min}|(\frac{3}{2}+50\sqrt{\varepsilon})x_{v^*}^2\right)>0.
\]
By Lemma \ref{lem-2-3}, we have $\lambda(G_5) > \lambda(G)$, a contradiction. 

Let $Q$ be the unique non-trivial component in $G[V_2]$. Next, we claim that $Q=K_{1,s}$ with $s\ge 3$. Otherwise, $Q\in\{C_3,K_{1,2},K_{1,1}\}$. It means that $Q_{min}\in \{K_{1,3},C_3,K_{1,2},K_{1,1}\}$. If $Q_{min}\ne K_{1,3}$,  let 
\[G_6=G-E(Q_{min})+E(V(Q_{min}),V(Q)).\]
Clearly, $G_6$ is still $C_{2k+1}^{\square}$-free. Moreover, from Lemma \ref{lem-3-13}, we have 
\[e(V(Q_{min}),V(Q))\le |V(Q_{min})|\cdot |V_Q|-(e(Q_{min})+1).\] 
Thus, we get \[\mathbf{x}^TA(G_6)\mathbf{x}-\mathbf{x}^TA(G)\mathbf{x}\geq (e(Q_{min})+1)(1-92\sqrt{\varepsilon})^2x^2_{v^*}-e(Q_{min})x^2_{v^*}>0.\] 
By Lemma \ref{lem-2-3}, we have $\lambda(G_6) > \lambda(G)$, a contradiction. If $Q_{min}=K_{1,3}$, then we could regard $Q$ as $Q_{min}$. Therefore, $K_{1,3}$ is the unique non-trivial component of $G[V_1]$. Thus, the claim holds by exchanging the index of $V_1$ and $V_2$.

Assume $V(Q)=V(K_{1,s})=\{v,v_1,\ldots,v_s\}$ with center $v$. If $v$ is not adjacent to any non-trivial component in $V_1$, then $G[V_1 \cup \{v'\}]$ is a union of triangles and stars, and $V_2 \setminus \{v'\}$ is an independent set. Therefore, applying the same method as in Case 1 shows that G must be isomorphic to $K_1\vee T_{n-1,2}$. It remains to consider the case that $v$ is adjacent some vertex $u\in V(H)$ for some $H\in \mathcal{X}^{(1)}$. 

Suppose that $H=C_3$. Assume $V(H)=\{u_0,u_1,u_2\}$ and $u=u_0$. Then $v_i\not\sim u_j$ for $1\le i\le s$ and $1\le j\le 2$ according to Lemma \ref{lem-3-8} (i). Let \[G_7=G-E(H)+\{v_iu_j\mid 1\le i\le s,1\le j\le 2\}.\]
Clearly, $G_7$ is $C_{2k+1}^{\square}$-free. Moreover, since $V(H)\cap W=\emptyset$, we have
\[\mathbf{x}^T(A(G_7)-A(G))\mathbf{x}\ge 2s(1-92\sqrt{\varepsilon})^2x_{v^*}^2-3x_{v^*}^2>0,\]
which implies $\lambda(G_7)>\lambda(G)$, a contradiction.

Suppose that $H=K_{1,t}$. Assume $V(H)=\{u_0,u_1,\ldots,u_t\}$ with center $u_0$. If $u=u_0$, then $v_i\not\sim u_j$ for $1\le i\le s$ and $1\le j\le t$ according to Lemma \ref{lem-3-8} (i). Let \[G_8=G-E(H)+\{v_iu_j\mid 1\le i\le s,1\le j\le t\}.\]
Clearly, $G_8$ is $C_{2k+1}^{\square}$-free. Moreover, since $|V(H)\cap W|\le 1$, we have
\[\mathbf{x}^T(A(G_8)-A(G))\mathbf{x}\ge st(1-92\sqrt{\varepsilon})^2x_{v^*}^2-t(\frac{3}{2}+50\sqrt{\varepsilon})x_{v^*}^2>0,\]
which implies $\lambda(G_8)>\lambda(G)$, a contradiction. Now we suppose that $v\not\sim u_0$. Therefore, $u\in\{u_1,\ldots,u_t\}$, say $u=u_1$. We show that there is no edges between $v$ and $H'$ for any $H'\in \mathcal{X}^{(1,*)}\setminus\{H\}$. Otherwise, assume there is an edge between $v$ and $H'=K_{1,r}$. Let $V(H')=\{u'_0,u'_1,\ldots,u'_r\}$ where $u'_0$ is the center. By arguments above, we already proved that $v\not\sim u'_0$. We may assume that $v\sim u'_1$. Let 
\[G_9=G-E(Q)+\{v_iu_0\mid 1\le i\le s\}+\{v_iu'_0\mid 1\le i\le s\}.\]
It is obvious that $G_9$ is $C_{2k+1}^{\square}$-free. Moreover, since $|V(Q)\cap W|\le 1$, we have
\[\mathbf{x}^T(A(G_9)-A(G))\mathbf{x}\ge 2s(1-92\sqrt{\varepsilon})^2x_{v^*}^2-s(\frac{3}{2}+50\sqrt{\varepsilon})x_{v^*}^2>0.\]
This implies $\lambda(G_9)>\lambda(G)$, a contradiction. Finally, if $x_{u_0}\ge x_v$,  let 
\[G_{10}=G-\{vw\mid w\in N_{V_2}(v)\}+\{u_0w\mid w\in N_{V_2}(v)\}.\]
Similarly, we get $G_{10}$ is $C_{2k+1}^{\square}$-free and $\lambda(G_{10})>\lambda(G)$, a contradiction. If $x_v\ge x_{u_0}$, we could yields the same construction by constructing $G_{11}=G-\{u_0w\mid w\in N_{V_2}(u_0)\}+\{vw\mid w\in N_{V_2}(u_0)\}$.

The proof is completed.
\end{proof}   

\section*{Acknowledgments}
This work is supported by NSFC (No. 12371362).

\section*{Declaration of competing interest}
We declare that we have no conflict of interest to this work.


{\small\begin{thebibliography}{99}
\bibitem{BJST2023}
D. Brada\u{c}, O. Janzer, B. Sudakov, I. Tomon, The Tur\'{a}n number of the grid, Bull. Lond. Math. Soc. 55 (2023) 194--204.

\bibitem{CDT}
S. Cioab\u{a}, D.N. Desai, M. Tait, The spectral even cycle problem, arXiv:2205 .00990.

\bibitem{CRS}D. Cvektovi\'{c}, P. Rowlinson, S. Simi\'{c}, An Introduction to the Theory of Graph Spectra, Cambridge University Press, New York, 2010.

\bibitem{CZ}
M. Chen, X. Zhang, Some new results and problems in spectral extremal graph theory, J. Anhui Univ. Nat. Sci. 42 (1) (2018) 12--25.

\bibitem{D1}
T. Dzido, A note on Tur\'{a}n numbers for even wheels, Graphs Comb. 29 (2013) 1305--1309.

\bibitem{DJ}
T. Dzido, A. Jastrzebski, Tur\'{a}n numbers for odd wheels, Discrete Math. 341 (4) (2018) 1150--1154.

\bibitem{DKLNTW}
D.N. Desai, L.Y. Kang, Y.T. Li, Z.Y. Ni, M. Tait, J. Wang, Spectral extremal graphs for intersecting cliques, Linear Algebra Appl., 644: 234--258, 2022.

\bibitem{ES}
P. Erd\H{o}s, M. Simonovits, Some extremal problems in graph theory, in: Combinatorial Theory and Its Applications, I, Proc. Colloq., Balatonf\"{u}red, 1969, North-Holland, Amsterdam, 1970, pp. 377--390.

\bibitem{ES1}
P. Erd\H{o}s, M. Simonovits, An extremal graph problem, Acta Math. Acad. Sci. Hung. 22 (1971--1972) 275--282.

\bibitem{F1}
Z. F\"{u}redi, On a theorem of Erd\H{o}s and Simonovits on graphs not containing the cube, in: Number Theory, Analysis, and Combinatorics, Proceedings of the Paul Tur\'{a}n Memorial Conference, Budapest, 2013, pp. 113--125.

\bibitem{FG}
Z. F\"{u}redi, D. Gunderson, Extremal numbers for odd cycles, Comb. Probab. Comput. 24 (4) (2015) 641--645.

\bibitem{HLF2025}
X. He, Y. Li, L. Feng, Extremal graphs for the odd prism, Discrete Math. 348 (2025) 114249.

\bibitem{LLF}
Y. Li, W. Liu, L. Feng, A survey on spectral conditions for some extremal graph problems, Adv. Math. (China) 51 (2) (2022) 193--258.

\bibitem{Nikiforov2007}
V. Nikiforov, Bounds on graph eigenvalues II, Linear Algebra Appl. 427 (2–3) (2007) 183--189.


\bibitem{Nikifrov2008}
V. Nikiforov, A spectral condition for odd cycles in graphs, Linear Algebra Appl. 428 (7) (2008) 1492--1498.

\bibitem{Niki1}
V. Nikiforov, The maximum spectral radius of $C_4$-free graphs of given order and size, Linear Algebra Appl. 430 (2009) 2898--2905.

\bibitem{N1}
V. Nikiforov, Some new results in extremal graph theory, in: Surveys in Combinatorics 2011, in: London Math. Soc. Lecture Note Ser., vol. 392, Cambridge Univ. Press, Cambridge, 2011, pp. 141--181.


\bibitem{J}
T. Jiang, A. Newman, Small dense subgraphs of a graph, SIAM J. Discrete Math. 31 (2017) 124--142.

\bibitem{JMY}
T. Jiang, J. Ma, L. Yepremyan, On Tur\'{a}n exponents of bipartite graphs, Comb. Probab. Comput. 31 (2022) 333--344.

\bibitem{Ore}
O. Ore, Arc coverings of graphs, Ann. Mat. Pura Appl. 55 (4) (1961) 315--321.

\bibitem{PS}
R. Pinchasi, M. Sharir, On graphs that do not contain the cube and related problems, Combinatorica 25 (2005) 615--623.

\bibitem{S1}
M. Simonovits, Extremal graph problems with symmetrical extremal graphs, additionnal chromatic conditions, Discrete Math. 7 (1974) 349--376.

\bibitem{S2}
M. Simonovits, The extremal graph problem of the icosahedron, J. Combin. Theory Ser. B 17 (1974) 69--79.
\bibitem{Y1}
L.-T Yuan, Extremal graphs for odd wheels, J. Graph Theory 98 (2021) 691--707.

\bibitem{ZW}
M. Zhai, B. Wang, Proof of a conjecture on the spectral radius of $C_4$-free graphs, Linear Algebra Appl. 437 (7) (2012) 1641--1647.

\end{thebibliography}}
\end{document}